\documentclass[12pt]{amsart}
\numberwithin{equation}{section}

\usepackage{comment}
\usepackage{bm}
\usepackage{mathtools}
\usepackage{enumitem}
\usepackage{microtype}
\usepackage[margin=1.15in]{geometry}
\usepackage{xcolor}
\usepackage{hyperref}
\usepackage{fancyhdr}
\usepackage{amsmath}             
\usepackage{amssymb}           
\usepackage{amsfonts}           
\usepackage{latexsym}           
\usepackage{amsthm}              
\usepackage{bbm}
\usepackage{tikz}
\usetikzlibrary{decorations.pathreplacing}
\usetikzlibrary{angles}
\usetikzlibrary{calc}
\usepackage[justification=centering]{caption}
\usepackage{type1cm}
\usepackage{todonotes}
\usepackage{standalone}
\reversemarginpar

\hypersetup{
colorlinks=true,
linkcolor=black,
anchorcolor=black,
citecolor=black,
filecolor=black,      
menucolor=red,
runcolor=black,
urlcolor=black,
}

\numberwithin{equation}{section}

\theoremstyle{plain}
\newtheorem{theorem}{Theorem}[section]

\renewcommand{\themaintheorem}{\Alph{maintheorem}}
\newtheorem{corollary}[theorem]{Corollary}
\newtheorem{proposition}[theorem]{Proposition}

\newtheorem{lemma}[theorem]{Lemma}

\theoremstyle{remark}
\newtheorem{remark}[theorem]{Remark}

\newtheorem{example}[theorem]{Example}

\theoremstyle{definition}

\newtheorem*{question*}{Question}

\newcommand{\notebox}[1]{\fcolorbox{red}{white}{\color{red}#1}}
\newcommand{\bnote}[1]{{\color{blue}#1}}
\newcommand{\rnote}[1]{{\color{red}#1}}

\newcommand{\xii}{\bm{\xi}}
\newcommand{\oomega}{\bm{\omega}}
\newcommand{\barf}{\overline{f}}
\newcommand{\loc}{\mathrm{loc}}
\newcommand{\qq}{\mathbf{q}}
\newcommand{\pp}{\mathbf{p}}
\newcommand{\LL}{\mathcal{L}}
\newcommand{\RR}{\mathcal{R}}
\newcommand{\GG}{\mathcal{G}}
\newcommand{\BB}{\mathcal{B}}
\newcommand{\HH}{\mathcal{H}}
\newcommand{\LLL}{\mathcal{L}}
\newcommand{\PP}{\mathcal{P}}
\newcommand{\MM}{\mathcal{M}}
\newcommand{\NN}{\mathcal{N}}
\newcommand{\EE}{\mathcal{E}}
\newcommand{\CC}{\mathcal{C}}
\newcommand{\TT}{\mathcal{T}}
\newcommand{\UU}{\mathcal{U}}
\newcommand{\OO}{\mathcal{O}}
\newcommand{\VV}{\mathcal{V}}
\newcommand{\R}{\mathbb{R}}
\newcommand{\RP}{\mathbb{RP}^1}
\newcommand{\C}{\mathbb{C}}
\newcommand{\Q}{\mathbb{Q}}
\newcommand{\Z}{\mathbb{Z}}
\newcommand{\N}{\mathbb{N}}
\newcommand{\hhh}{\mathtt{h}}
\newcommand{\iii}{\mathtt{i}}
\newcommand{\jjj}{\mathtt{j}}
\newcommand{\kkk}{\mathtt{k}}
\renewcommand{\lll}{\mathtt{l}}

\newcommand{\aaa}{\mathtt{a}}
\newcommand{\bbb}{\mathtt{b}}

\newcommand{\eps}{\varepsilon}
\newcommand{\fii}{\varphi}
\newcommand{\roo}{\varrho}
\newcommand{\ualpha}{\overline{\alpha}}
\newcommand{\lalpha}{\underline{\alpha}}
\newcommand{\Wedge}{\textstyle\bigwedge}
\newcommand{\la}{\langle}
\newcommand{\ra}{\rangle}
\newcommand{\A}{\mathrm{A}}
\newcommand{\B}{\mathsf{B}}
\newcommand{\F}{\mathsf{F}}
\newcommand{\FF}{\mathcal{F}}
\newcommand{\pv}{\underline{p}}
\newcommand{\dd}{\,\mathrm{d}}
\newcommand{\projspan}[1]{\left\langle#1\right\rangle}
\newcommand{\Bad}{\mathbf{Bad}}
\newcommand{\Good}{\mathbf{Good}}
\newcommand{\TSigma}{\widetilde{\Sigma}}
\newcommand{\TLambda}{\widetilde{\Lambda}}
\newcommand{\TK}{\widetilde{K}}

\renewcommand{\ge}{\geqslant}
\renewcommand{\le}{\leqslant}
\renewcommand{\geq}{\geqslant}
\renewcommand{\leq}{\leqslant}

\DeclareMathOperator{\dimloc}{dim_{loc}}
\DeclareMathOperator{\udimloc}{\overline{dim}_{loc}}
\DeclareMathOperator{\ldimloc}{\underline{dim}_{loc}}

\DeclareMathOperator{\dimh}{dim_H}

\DeclareMathOperator{\dima}{dim_A}

\DeclareMathOperator{\dist}{dist}
\DeclareMathOperator{\diam}{diam}

\DeclareMathOperator{\spt}{spt}

\newlist{aenumerate}{enumerate}{1}
\setlist[aenumerate, 1]{label=(A\arabic{aenumeratei})}

\title[Multifractal analysis of pointwise Assouad dimension]{
Multifractal analysis for the pointwise Assouad dimension of self-similar measures
}

\author{Roope Anttila}
\address{Research Unit of Mathematical Sciences\\  P.O.Box 8000, FI-90014, University of Oulu, Finland }
\email{roope.anttila@oulu.fi}
\author{Ville Suomala}
\address{Research Unit of Mathematical Sciences\\  P.O.Box 8000, FI-90014, University of Oulu, Finland }
\email{ville.suomala@oulu.fi}

\date{\today}
\thanks{RA is supported by the Magnus Ehrnrooth foundation}

\begin{document}

\begin{abstract}
We quantify the pointwise doubling properties of self-similar measures using the notion of pointwise Assouad dimension. We show that all self-similar measures satisfying the open set condition are pointwise doubling in a set of full Hausdorff dimension, despite the fact that they can in general be non-doubling in a set of full Hausdorff measure. More generally, we carry out multifractal analysis by determining the Hausdorff dimension of the level sets of the pointwise Assouad dimension.
\end{abstract}

\maketitle
\section{Introduction}
In analysis and geometry, one often encounters situations where indirect information about size and structure is obtained by comparing a family of balls to the family obtained by the dilations of the balls by a certain ratio. The classical $5R$-covering theorem is a prototype case: One can extract a packing from an arbitrary collection of balls which still covers the collection after dilating the radii by a uniform constant. For measures, analogous and often finer tools are available through the use of Vitali type covering theorems, which are on the other hand intimately related to the doubling properties of the measure.

A Borel measure $\mu$ supported on a metric space $X$, and assigning finite measure to bounded sets, is \emph{doubling}, if there is a constant $C>0$ such that for any $x\in X$ and $r>0$, we have 
\begin{equation}\label{eq:doubling-general}
  \mu(B(x,2r))\leq C\mu(B(x,r))<\infty,
\end{equation}
where $B(x,r)$ is the closed ball of radius $r>0$ and center $x$. This condition is a mild regularity requirement which, in many cases, streamlines analysis on $X$. For instance, the above mentioned Vitali covering theorem, and hence the Lebesgue differentiation theorem hold for all doubling measures \cite{Heinonen2001} providing a natural framework e.g. for harmonic analysis and analysis on metric spaces (see  \cite{Stein1993, Heinonen2001, MackayTyson2010, HeinonenKoskelaShanmugalingamTyson2015, BBL2017}). For finer analysis, the global condition \eqref{eq:doubling-general} is often insufficient since most measures are not doubling and even if $\mu$ is doubling the estimate \eqref{eq:doubling-general} will most likely be sub-optimal for some/many/typical $x\in X$. 

For a finer study of the doubling properties of measures, we say that a measure $\mu$ is 
\emph{(pointwise) doubling at $x\in X$}, if there is some $0<C(x)<\infty$ such that $\mu(B(x,2r))\le C(x)\mu(B(x,r))<\infty$ for all $r>0$. This is easily seen to be equivalent to the condition that for all $\gamma>1$,
\begin{equation}\label{eq:pointwise-doubling}
C_\gamma(x):=\inf_{r>0}\frac{\mu(B(x,\gamma r))}{\mu(B(x,r))}<\infty\,.
\end{equation}
Of course, a doubling measure is pointwise doubling at all points of its support but the converse need not be true, as pointed out for example in \cite[Example 3.3]{Anttila2022}. Relaxing the global doubling condition to the condition that $\mu$ is pointwise doubling at $\mu$-almost every $x$ is enough e.g. for the Vitali covering theorem to hold \cite[Theorem 3.4.4]{HeinonenKoskelaShanmugalingamTyson2015} so much of analysis can be carried out in this larger class of metric measure spaces. This also motivates the following general question: What is the size of the set where a given measure $\mu$ is pointwise doubling?

The goal of this paper is to answer the previous question and conduct a fine analysis of the doubling properties for self-similar measures, where the defining iterated function system satisfies the open set condition (OSC). Under the more restrictive strong separation condition, all self-similar measures are doubling, but under the OSC the situation is more subtle. A pedagogical example is given by the $(p,1-p)$-Bernoulli measure attached to the system $\{x\mapsto x/2,x\mapsto1/2+x/2\}$, which is doubling if and only if $p=1/2$. For a general self-similar iterated function system with OSC, the so-called canonical self-similar measure is always doubling \cite{MauldinUrbanski1996}, but depending on the situation, there could be (1) No other doubling self-similar measures, (2) Infinitely many doubling and infinitely many non-doubling self-similar measures, (3) No non-doubling self-similar measures. Yung \cite{Yung2007} provides a detailed analysis of such different situations along with many examples. 

For a given measure $\mu$, we let
\begin{equation*}
  D(\mu)=\{x\in\spt(\mu)\colon \mu\text{ is doubling at }x\}\,,
\end{equation*}
and let $\dimh X$ denote the Hausdorff dimension of $X$. If a measure $\mu$ is doubling, then $D(\mu)=X$. We are interested in the size of the set $D(\mu)$ for self-similar measures $\mu$ satisfying the OSC in the non-trivial case when $\mu$ is not doubling. In our first theorem,which is an immediate corollary of our main result, we show that in this context, the set $D(\mu)$ has full Hausdorff dimension but zero Hausdorff measure.

\begin{theorem}\label{thm:main-doubling}
  Let $\mu$ be a non-doubling self-similar measure fully supported on a self-similar set $X\subset\mathbb{R}^n$ satisfying the OSC and let $s=\dimh X$. Then
  \begin{enumerate}
    \item $\dimh D(\mu)=s$,\label{thm:main-1}
    \item $\HH^{s}(D(\mu))=0$.\label{thm:main-2}
  \end{enumerate}
\end{theorem}

To state our main result, we recall the notion of \emph{pointwise Assouad dimension for measures} which was recently introduced in \cite{Anttila2022}. For a measure $\mu$ supported on a metric space $X$, the pointwise Assouad dimension of $\mu$ at $x\in \spt(\mu)$ is defined by
\begin{align*}
  \dima(\mu,x)=\inf\bigg\{s>0\colon& \exists C>0,\text{ s.t. }\forall 0<r<R,\\
  &\frac{\mu(B(x,R))}{\mu(B(x,r))}\leq C\left(\frac{R}{r}\right)^s\bigg\}\,.
\end{align*}
Let us point out that if $C$ is not allowed to depend on the point $x$, the corresponding infimum gives the (global) Assouad dimension, $\dima\mu$, which is sometimes also referred to as the upper regularity dimension \cite{KaenmakiLehrbackVuorinen2013,Fraser2020}.
A related pointwise notion for sets was recently introduced and studied by Käenmäki and Rutar \cite{KaenmakiRutar2023+}.

The pointwise Assouad dimension at $x$ is finite if and only if the measure is doubling at $x$. Indeed, by modifying the proof of Theorem 2.2 in \cite{Howroyd2019} in the obvious way, it is easy to see that $\dima(\mu,x)$ quantifies the smallest \emph{doubling exponent} of $\mu$ at a given point, that is
\begin{equation*}
    \dima(\mu,x)=\inf_{\gamma>1}\frac{\log C_{\gamma}(x)}{\log \gamma}\,.
  \end{equation*} 
In \cite{Anttila2022} it was shown that many classical fractal measures, such as the self-similar measures we study in this paper, exhibit an exact dimensionality property for the pointwise Assouad dimension, that is $\dima(\mu,x)=\dima \mu$ for $\mu$-almost every $x\in\spt(\mu)$. The purpose of this paper is to take the analysis further and study the non-typical behaviour via multifractal analysis.

Classically, in multifractal analysis one is interested in the size of the level sets of the \emph{pointwise dimension} of a given measure $\mu$, which is defined at $x\in X$ as
\begin{equation}\label{eq:dimloc}
  \dimloc(\mu,x)=\lim_{r\downarrow 0}\frac{\log \mu(B(x,r))}{\log r}\,,
\end{equation}
provided the limit exists; if the limit does not exist, we denote by $\ldimloc(\mu,x)$ and $\udimloc(\mu,x)$ the \emph{lower and upper pointwise dimensions}, which one obtains by replacing the limits in \eqref{eq:dimloc} by the lower and upper limits, respectively. Given a measure $\mu$, we denote the $\alpha$-level set of the pointwise dimension by
\begin{align}\label{eq:Da}
  &D_\alpha^\loc(\mu)=\{x\in\spt(\mu)\colon \dimloc(\mu,x)=\alpha\},
\end{align}
where we may omit $\mu$ from the notation if the measure is clear from the context. The general objective in multifractal analysis is then to determine the \emph{multifractal spectrum}, $\alpha\mapsto\dimh D_\alpha^{\loc}$ for the pointwise dimension map. 
It was first observed by physicists studying turbulence \cite{FrischParisi1985} that for certain well behaved measures, the multifractal spectrum seems to behave in a regular way. 
For self-similar measures under the OSC, the multifractal spectrum is nowadays very well understood (\cite{CawleyMauldin1992, ArbeiterPatzschke1996, Patzschke1997, Falconer1997}); Let $\tau_\mu(q)$ denote the
$L_q$-spectrum of $\mu$.
Then, there are numbers $0<\alpha_{\min}\le\alpha_{\max}<\infty$ such that all values of $\dimloc(\mu,x)$ fall onto the interval $[\alpha_{\min},\alpha_{\max}]$ and for $\alpha_{\min}\le \alpha\le \alpha_{\max}$, the value of $\dimh D_\alpha^{\loc}$ equals the Legendre transform
\begin{equation}\label{eq:Legendre}
  f(\alpha)\coloneqq\inf_{q\in\R}\alpha q+\tau_{\mu}(q).
\end{equation}
There is also a well known analytic expression for $f(\alpha)$, which we recall in Section \ref{sec:classical} (see \eqref{eq:alpha-q-correspondence}). It is implicitly required in the above result that the limit in \eqref{eq:dimloc} exists. However, it is also well known \cite{CawleyMauldin1992} that the Hausdorff dimension of the $\alpha$-sublevel set
\begin{equation*}
  U_{\alpha}^{\loc}(\mu)\coloneqq \{x\in\spt(\mu)\colon \udimloc(\mu,x)\leq\alpha\}
\end{equation*}
coincides with the value $\barf(\alpha)=\max_{\beta\leq \alpha}f(\alpha)$, for all $\alpha_{\min}\le \alpha\le \alpha_{\max}$.

From the point of view of pointwise Assouad dimension, one is led to ask what are the sizes of the level sets
\begin{align*}
  &D_{\alpha}^\A(\mu)=\{x\in \spt(\mu)\colon \dima(\mu,x)=\alpha\},
\end{align*}
or of the sub-level sets
\begin{align*}
  &U_{\alpha}^\A(\mu)=\{x\in \spt(\mu)\colon \dima(\mu,x)\leq\alpha\},
\end{align*}
of $\dima$ for all values of $\alpha$. 
In our main result, we completely solve this question for self-similar measures under the OSC.
\begin{theorem}\label{thm:main-spectrum}
  Let $\mu$ be a self-similar measure satisfying the OSC and let $X$ be the associated self-similar set. Then, for all $\alpha\in[\alpha_{\min},\alpha_{\max})$,
  \begin{equation*}
    \dimh D_{\alpha}^\A(\mu)=\dimh U_{\alpha}^\A(\mu)
    =\barf(\alpha).
  \end{equation*}
  Moreover,
  \begin{equation*}
    \dimh D_{\dima\mu}^\A(\mu)=\dimh U_{\dima\mu}^\A(\mu)
    =\dimh X,
  \end{equation*}
  and for any $\alpha\not\in [\alpha_{\min},\alpha_{\max}]\cup\{\dima\mu\}$, we have $D_\alpha^\A(\mu)=\emptyset$.
\end{theorem}

If $\mu$ is doubling, then by \cite[Theorem 2.4]{FraserHowroyd2019}, $\dima\mu=\alpha_{\max}$ and therefore our theorem tells that the Hausdorff spectrum of the pointwise Assouad dimension fully coincides with the upper spectrum of the pointwise dimension. Moreover, even if the measure is not doubling, the spectra coincide outside of the isolated point at $+\infty$ and possibly the point $\alpha_{\max}$ where many types of behaviour is possible, see Remark \ref{rem:alpha_max}. In spite of the similarity in the behaviour of the multifractal spectra, the method for obtaining the result for the pointwise Assouad dimension is quite different from the classical approach for the pointwise dimension.

Let us outline the structure of the paper. In Section \ref{sec:preliminaries} we recall our basic definitions and notations and outline the classical approach for obtaining the multifractal spectrum for self-similar measures. We also briefly discuss why the classical approach does not work for the pointwise Assouad dimension. In Section \ref{sec:symbolic} we develop methods for determining the multifractal spectrum for the pointwise Assouad dimension in the symbolic setting, and in Section \ref{sec:OSC}, we apply these methods to study self-similar measures under the OSC and prove Theorems \ref{thm:main-doubling} and \ref{thm:main-spectrum}.

\section{Preliminaries}\label{sec:preliminaries}
From now on, a measure will refer to an outer regular Borel measure. In most cases, these will be finite or even probability measures. We will denote by $C$ or $C(\cdots)$ various constants where, whenever needed, the quantities inside the parentheses explain the dependency of $C$ on other parameters. For $A\subset X$, where $X$ is a metric space, $\dimh A$ denotes the Hausdorff dimension of $A$. For a measure $\mu$, $\dimh\mu$ is the infimal Hausdorff dimension of sets of positive $\mu$ measure. We denote closed balls in $X$ with centre $x$ and radius $r>0$ by $B(x,r)$.

\subsection{Self-similar sets and measures}
Let $m\in \N$ and $\Lambda = \{1,\ldots,m\}$. A \emph{self-similar iterated function system (IFS)} is a collection $(\varphi_i)_{i\in\Lambda}$ of contractive similarities on $\R^d$. We denote the contraction ratio of each $\varphi_i$ by $r_i$.  Let $X$ denote the \emph{self-similar set} associated to $(\varphi_i)_{i}$. This is the unique compact and non-empty $X\subset\R^d$ satisfying
\begin{equation*}
X=\bigcup_{i\in\Lambda}\varphi_i(X).
\end{equation*}
By rescaling $X$ if necessary, we assume that $\diam(X)=1$. As a metric space, $X$ inherits the Euclidean metric from the ambient Euclidean space $\R^d$.
We say that the self-similar set $X$ satisfies the \emph{strong separation condition (SSC)} if $\varphi_i(X)\cap \varphi_j(X)=\emptyset$, for all $i\ne j$ and we say that $X$ satisfies the \emph{open set condition (OSC)} if there exists a non-empty bounded open set $U\subset \R^d$ satisfying $\varphi_i(U)\subset U$, for all $i\in\Lambda$ and $\varphi_i(U)\cap\varphi_j(U)=\emptyset$, for all $i\ne j$. 

We denote the (open) simplex of probability vectors in $\R^m$ by
\begin{equation*}
  \PP = \left\{\pp\coloneqq (p_1,\ldots,p_m)\in\R^m\colon \sum_{i=1}^mp_i=1, \text{ and } p_i>0,\forall i=1,\ldots,m\right\}.
\end{equation*}
For every $\pp\in\PP$, there is a unique Borel probability measure $\mu_\pp$ fully supported on $X$, which satisfies
\begin{equation}\label{eq:ss_meas_def}
\mu_\pp = \sum_{i\in\Lambda} p_i\varphi_i\mu_\pp\,,
\end{equation}
where $f\mu$ denotes the push-forward measure $\mu\circ f^{-1}$, for any measurable function $f$. 
This measure is called the \emph{self-similar measure} associated to the IFS $(\varphi_1,\ldots,\varphi_m)$ and the weight vector $\pp$.

\subsection{Symbolic spaces and Bernoulli measures}\label{sec:sym_B}
Let $\Lambda$ be a finite index set. The \emph{symbolic space} generated by the alphabet $\Lambda$ is denoted by $\Sigma(\Lambda)\coloneqq\Lambda^{\N}$.
If the alphabet is clear from the context, we suppress $\Lambda$ from the notation. The elements of $\Sigma$ are called \emph{(infinite) words}. For a given $n\in\N$, the set of \emph{words of length $n$} is denoted by $\Sigma_n\coloneqq\Lambda^n$ and the set of all \emph{finite words} by $\Sigma_*=\bigcup_{n\in\N} \Sigma_n$. The following notation is used both for the infinite words $\iii=(i_1,i_2,\ldots)\in\Sigma$ as well as for finite words $\iii=(i_1,\ldots,i_k)\in\Sigma_*$, when appropriate. Additionally, we reserve the letters $\aaa$ and $\bbb$ exclusively for finite words to avoid confusion in certain cases. By $|\iii|\in\N\cup\{\infty\}$ we denote the length of the word $\iii$. The projection of $\iii$ to the first $n\le|\iii|$ letters is $\iii|_n=(i_1,\ldots,i_n)$. Given two words $\iii$ and $\jjj$, $\iii\wedge\jjj$ is their longest common prefix. We write $\jjj\sqsubset \iii$ if $\jjj\in\Sigma_*$ is a \emph{subword} of $\iii\in\Sigma$. That is, for some $n\in\N$, $\sigma^n\iii|_{|\jjj|}=\jjj$, where $\sigma\colon \Sigma\to\Sigma$ is the left shift on $\Sigma$ defined by $\sigma(i_1,i_2,i_3,\ldots)=(i_2,i_3,\ldots)$. For a finite word $\iii$, we also let $\iii^-$ denote the word obtained from $\iii$ by deleting the last symbol.

Self-similar sets and measures admit a natural symbolic coding. Suppose $(\varphi_i)_{i\in\Lambda}$ is a self-similar IFS as in the previous subsection and let $\Lambda = \{1,\ldots,m\}$. There is a natural coding map $\pi$ between the symbolic space $\Sigma$ and the self-similar set $X$, defined by
\begin{equation}\label{eq:projection}
\pi(\iii)=\lim_{n\to \infty} \varphi_{\iii|_n}(0).
\end{equation}
Under the SSC, the map $\pi$ is a bijection but $\pi$ may fail to be injective if only the OSC is assumed. We next define a metric on $\Sigma$, which allows us to capture some geometric properties of $X$ on the symbolic side. To that end, given $\iii\in\Sigma_n$, we denote
\begin{align*}
  &r_{\iii}=r_{i_1}r_{i_2}\cdots r_{i_n},\\
  &p_{\iii}=p_{i_1}p_{i_2}\cdots p_{i_n},\\
  &\varphi_{\iii}=\varphi_{i_1}\circ\varphi_{i_2}\circ\ldots\circ \varphi_{i_n}.
\end{align*}
Let
\begin{equation*}
  \rho(\iii,\jjj)=r_{\iii\wedge\jjj},
\end{equation*}
It is easy to see that $\rho$ is a metric (in fact an ultrametric) on $\Sigma$ and under the metric $\rho$ the projection map $\pi\colon \Sigma\to X$ is Lipschitz and if $X$ satisfies the SSC it is even bi-Lipschitz.

Given $\iii\in\Sigma_*$, we denote the corresponding cylinder set by $[\iii]\coloneqq \{\jjj\in\Sigma\colon \jjj|_{|\iii|}=\iii\}$. These cylinder sets are clopen and form a basis for the topology induced by $\rho$. Indeed, for any $\iii\in\Sigma$ and $r>0$, we have
\begin{equation}\label{eq:balls-are-cylinders}
  B(\iii,r)=[\iii|_n],
\end{equation}
where 
$n=n(\iii,r)\in\N$ is the unique natural number satisfying $r_{\iii|_{n}}\le r < r_{\iii|_{n-1}}$.

Finally, we will consider Bernoulli measures on $\Sigma$. For $\pp\in\PP$, we define a pre-measure $\nu_\pp$ on the cylinder sets in $\Sigma$ by letting
\begin{equation*}
  \nu_\pp([\iii])=p_{\iii},
\end{equation*}
for all $\iii\in\Sigma_*$ and extend it to a unique Borel probability measure on $\Sigma$. The measure $\nu_\pp$ is called the \emph{Bernoulli measure} associated to the probability vector $\pp$. Under the OSC it holds that $\mu_\pp=\pi\nu_\pp$.

\subsection{Multifractal analysis of pointwise dimensions}\label{sec:classical}
Let us recall how to obtain an analytic expression for the multifractal spectrum of the pointwise dimension in the simplest nontrivial setting i.e. for Bernoulli measures on the symbolic space. Let $\nu$ and $\rho$ along with the $r_i$ and $p_i$, $i\in\Lambda$, be as in Section \ref{sec:sym_B} above. For any $q\in\R$, we define $\tau(q)$ as the unique real number satisfying
\begin{equation*}
  \sum_{i=1}^mp_i^qr_i^{\tau(q)}=1\,.
\end{equation*}
The function $\tau\colon\R\to\R$ is strictly decreasing, continuous and strictly convex on the interval $[\alpha_{\min},\alpha_{\max}]$, where $\alpha_{\min}=\min_{i\in\Lambda}\tfrac{\log p_i}{\log r_i}$, $\alpha_{\max}=\max_{i\in\Lambda}\tfrac{\log p_i}{\log r_i}$. To avoid degenerate cases, we assume that $\alpha_{\max}>\alpha_{\min}$. Let $f(\alpha)=\inf_{q\in\R}\alpha q+\tau(q)$ denote the Legendre transform of $\tau$. By the analytic implicit function theorem, for each $\alpha\in (\alpha_{\min},\alpha_{\max})$ there exists a unique $q=q(\alpha)$, such that 
\begin{equation}\label{eq:alpha-q-correspondence}
  f(\alpha)=\alpha q+\tau(q)\,,
\end{equation}
and it follows from the convexity of $\tau$ that $\alpha=\alpha(q)$, as defined via \eqref{eq:alpha-q-correspondence} is a decreasing function of $q\in\R$. Using implicit differentiation leads to the expression 
\begin{equation*}
  \alpha=\alpha(q)
  =\frac{\sum_{i\in\Lambda}p_i^qr_i^{\tau(q)}\log p_i}{\sum_{i\in\Lambda}p_i^qr_i^{\tau(q)}\log r_i}.
\end{equation*}
Simple calculations show that $\alpha(q)\to \alpha_{\min}$ as $q\to \infty$ and $\alpha(q)\to \alpha_{\max}$ as $q\to -\infty$. Combining the above, for $q\in\R$, $\alpha=\alpha(q)$, we have
\begin{equation}\label{eq:mf_expl}
  f(\alpha)=q\frac{\sum_{i\in\Lambda}p_i^qr_i^{\tau(q)}\log p_i}{\sum_{i\in\Lambda}p_i^qr_i^{\tau(q)}\log r_i} + \tau(q)=\frac{\sum_{i\in\Lambda}p_i^qr_i^{\tau(q)}\log p_i^qr_i^{\tau(q)}}{\sum_{i\in\Lambda}p_i^qr_i^{\tau(q)}\log r_i}\,.
\end{equation}
Note that $f(\alpha)$ is maximised when $q=0$ and that $f(\alpha(0))=\tau(0)=\dimh \Sigma$.  Moreover, $f(\alpha)$ is strictly increasing on $(\alpha_{\min},\alpha(0)]$ and strictly decreasing on $[\alpha(0),\alpha_{\max})$ so we may set $f(\alpha_{\min})=\lim_{\alpha\to\alpha_{\min}}f(\alpha)$ and $f(\alpha_{\max})=\lim_{\alpha\to\alpha_{\max}}f(\alpha)$. Let us set $\overline{f}(\alpha)=\max_{\beta\leq\alpha}f(\beta)$, or more explicitly,
\begin{equation}
  \overline{f}(\alpha)=\begin{cases}
      f(\alpha),&\text{ if }\alpha_{\min}\le \alpha\leq \alpha(0)\\
      \dimh\Sigma,&\text{ if }\alpha_{\max}\ge \alpha> \alpha(0)
  \end{cases}.
\end{equation}
The following is a well known theorem, which establishes the multifractal formalism in the symbolic setting, see e.g. \cite{Falconer1997}.
In addition to the level sets $D_{\alpha}^\loc(\nu)$, recall \eqref{eq:Da}, we consider the sublevel sets $U_\alpha^\loc(\nu)=\{x\in\spt(\nu)\colon \udimloc(\nu,x)\leq\alpha\}$.
\begin{theorem}\label{thm:classical-formalism}
  For any $\alpha\in[\alpha_{\min},\alpha_{\max}]$, we have
  \begin{equation*}
    \dimh D_{\alpha}^\loc(\nu)=f(\alpha)
  \end{equation*}
  and
    \begin{equation*}
    \dimh U_{\alpha}^\loc(\nu)=\barf(\alpha)\,.
  \end{equation*}
  Furthermore, $D_{\alpha}^\loc(\nu)=\emptyset$, for $\alpha\not\in[\alpha_{\min},\alpha_{\max}]$.
\end{theorem}
 
This, now standard, theorem holds also for self-similar measures under the OSC \cite{CawleyMauldin1992,ArbeiterPatzschke1996} but the above symbolic version is enough for our considerations.
Let us briefly recall the main idea of the proof; for more details, see e.g. \cite{CawleyMauldin1992, Falconer1997}. For a given $\alpha\in[\alpha_{\min},\alpha_{\max}]$ a new Bernoulli measure $\nu_\alpha$ is generated by the probability vector $(p_1^q r_1^{\tau(q)},\ldots,p_m^q r_m^{\tau(q)})$, where $q=q(\alpha)$ is given by \eqref{eq:alpha-q-correspondence}. Then, one shows that $\nu_{\alpha}(D_{\alpha}^{\loc}(\nu))=1$ and further $\dimloc(\nu_\alpha,\iii)=f(\alpha)$ for all $\iii\in D_{\alpha}^{\loc}(\nu)$.   In the context of the pointwise Assouad dimension the above method is doomed to fail due to the following theorem. 
\begin{theorem}\label{thm:exact-dim}
  Let $\mu$ be a doubling self-similar measure supported on a self-similar set $X$ satisfying the OSC, and let $\nu$ be an ergodic measure on $X$. Then
  \begin{equation*}
    \dima(\mu,x)=\dima\mu=\alpha_{\max},
  \end{equation*}
  for $\nu$-almost every $x\in X$.
\end{theorem}
\begin{proof}
  In \cite{Anttila2022}, Lemma 4.7 holds for any ergodic measure, and therefore \cite[Theorem 4.1]{Anttila2022} holds in the above form.
\end{proof}
This theorem indicates a drastic difference in the typical behaviour of $\dimloc(\mu,x)$ and $\dima(\mu,x)$ for self-similar measures. In particular, for \emph{any} self-similar measure $\nu$, one has $\nu(D_\alpha^\A(\mu))=0$, for all $\alpha<\alpha_{\max}$. Since we cannot employ the standard measure theoretic tools, we have to develop a more hands on approach. We will first show, by employing the method of types approach as in \cite{Rutar2023a}, that in the symbolic setting the dimension of the sublevel sets $\{x\,:\,\dima(\mu,x)\le\alpha\}$ is given by $\barf(\alpha)$. Constructing a large Moran subset of $D_{\alpha}^\A$, we then pass from the the sublevel sets to the actual multifractal spectrum $\dimh D_{\alpha}^\A$. After this, passing to a subsystem with additional separation, we transfer the symbolic results to the geometric setting of self-similar measures under the OSC. If the self-similar measure is doubling, this is relatively easy. In the non-doubling case, the main obstacle is to show that outside of the full measure set of points where $\dima(\mu,x)=+\infty$, there is a full dimensional subset of $X$  where the multifractal spectrum is captured on the symbolic side. Finally, we show that even if $\mu$ fails to be doubling, there are no points where the pointwise Assouad dimension is in $(\alpha_{\max},+\infty)$.

\section{Multifractal analysis of Bernoulli measures}\label{sec:symbolic}
In this section, we conduct multifractal analysis for the pointwise Assouad dimension of Bernoulli measures on symbolic spaces. The results of this section serve as a basis for the multifractal analysis of self-similar measures conducted later in Section \ref{sec:OSC}.  Recently, an approach to multifractal analysis via Lagrange duality was developed by Rutar \cite{Rutar2023a} and some of his ideas, such as the method of types from large deviation theory, turn out to be useful for our purposes.

For the remainder of the section, we fix a symbolic space $\Sigma$ and a probability vector $\pp\in\PP$ along with the contraction parameters $(r_1,\ldots,r_m)$ which determine our metric $\rho$. We will denote by $\nu=\nu_\pp$ the associated Bernoulli measure and simplify the notation $D_{\alpha}^\A(\nu)$ to 
\begin{align*}
  &D_{\alpha}^\A=\{x\in \spt(\nu)\colon \dima(\mu,x)=\alpha\}\,.
\end{align*}
We also let
\begin{align*}
  &U_{\alpha}^\A=\{x\in \spt(\nu)\colon \dima(\mu,x)\leq\alpha\}\,.
\end{align*}
The main result of this section establishes the full multifractal spectrum for the Bernoulli measures.
\begin{theorem}\label{thm:symbolic-multifractal-spectrum}
  For any $\alpha\in[\alpha_{\min},\alpha_{\max}]$, we have
  \begin{equation*}
      \dimh D_{\alpha}^\A=\dimh U_{\alpha}^\A=\overline{f}(\alpha).
  \end{equation*}
  Moreover, for $\alpha\not\in[\alpha_{\min},\alpha_{\max}]$, we have $D_{\alpha}^\A = \emptyset$.
\end{theorem}

\subsection{Method of types}\label{sec:types}
For $\aaa\in\Sigma_*$ and $i\in\Lambda$, we let $N_i(\aaa)$ denote the number of times the symbol $i$ appears in $\aaa$ and set
\begin{equation*}
  \xii_i(\aaa)=\frac{N_i(\aaa)}{|\aaa|}\,.
\end{equation*}
We denote the \emph{type of $\aaa$} by $\xii(\aaa)\coloneqq(\xii_1(\aaa),\ldots,\xii_{m}(\aaa))\in\PP$. Given a collection $\Gamma\subset \Sigma_*$, we let
\begin{equation*}
  \PP_\Gamma=\left\{\xii(\aaa)\in\PP\colon \aaa\in\Gamma\right\}.
\end{equation*}
For $\Gamma\subset \Sigma_*$ and $\qq\in\PP_{\Gamma}$, we define the \emph{type class} of $\qq$ by
\begin{equation}\label{eq:Tn_def}
T_\Gamma(\qq)=\{\aaa\in\Gamma\colon \xii(\aaa)=\qq\}.
\end{equation}
Simplifying the notation slightly, we denote $T_n(\qq)=T_{\Sigma_n}(\qq)$. The main benefit in introducing type classes is that the quantities $r_{\aaa}$ and $p_{\aaa}$ are constant on each type class. We also remark that the type classes form a partition of $\Gamma$, i.e. $\Gamma=\bigcup_{\qq\in\PP_{\Gamma}}T_\Gamma(\qq)$.

We say that a sequence $(\Gamma_n)_{n\in\N}$, where $\Gamma_n\subset \Sigma_n$ is \emph{abundant}, if the following conditions are satisfied:
\begin{aenumerate}
  \item There is a constant $C>0$, such that $\#T_{\Gamma_n}(\qq)\geq C\#T_n(\qq)$, for all $n\in\N$ and $\qq\in\PP_{\Gamma_n}$.\label{it:abundant-1}
  \item For every $\delta>0$, the set $\PP_{\Gamma_n}$ is $\delta$-dense in $\PP$ for all large enough $n\in\N$.\label{it:abundant-2}
\end{aenumerate}
In Section \ref{sec:OSC}, the absence of geometric separation of cylinders causes difficulties. Using abundant sequences, and the induced alphabets, we will be able to overcome these and transfer the computation of the pointwise Assouad dimension to the symbolic side. 

We recall the following basic fact, which follows from (A1) and \cite[Lemma 2.3]{Csiszar2011}.
\begin{lemma}\label{lemma:T_n-cardinality-variant}
  If $(\Gamma_n)_n$ is abundant, then for any $n\in\N$ and $\qq\in\PP_{\Gamma_n}$, we have
  \begin{equation*}
      nH(\qq) - \log C(n+1)^{m+1}\leq\log\#T_{\Gamma_n}(\qq)\leq nH(\qq).
  \end{equation*}
\end{lemma}

Classically, the pointwise dimensions of self-similar measures are described by the entropy of the measure $\mu$ and the Lyapunov exponent of the associated dynamical system, the latter of which captures the contribution of the geometry of the projection map $\pi$ at typical points. For probability vectors $\qq\in\PP$ and $\pp\in\PP$ (the latter of which we have fixed), we denote by 
\begin{equation*}
H(\qq)=-\sum_{i=1}^m q_i\log q_i, \text{ and}\quad H_\pp(\qq)=-\sum_{i=1}^m q_i\log p_i,
\end{equation*}
the \emph{entropy} of $\qq$, and the \emph{cross entropy} of $\pp$ relative to $\qq$, respectively. The \emph{Lyapunov exponent} of $\qq$ is defined by
\begin{equation*}
\lambda(\qq)=-\sum_{i=1}^m q_i\log r_i.
\end{equation*}
These notions are related to the pointwise Assouad dimension of Bernoulli measures via the following characterisation which we will use repeatedly in what follows. The lemma follows easily from \eqref{eq:balls-are-cylinders} and the definition of the Bernoulli measure.
\begin{lemma}\label{lemma:explicit-assouad}
  Let $p\in\PP$ and $\nu$ be the Bernoulli measure on $\Sigma$ associated to $p$. Then for any $\iii\in\Sigma$ we have
  \begin{align*}
  \dima(\nu,\iii)&=\limsup_{n\to\infty}\sup_{\aaa\sqsubset \iii,|\aaa|=n}\frac{\log p_{\aaa}}{\log r_{\aaa}}=\limsup_{n\to\infty}\sup_{\aaa\sqsubset \iii,|\aaa|=n}\frac{H_{\pp}(\xii(\aaa))}{\lambda(\xii(\aaa))}
  \end{align*}
\end{lemma}

\subsection{Multifractal analysis}
Since we cannot employ the standard measure theoretic tools for multifractal analysis, we will commence by building large Moran sets inside the level sets $D_\alpha^\A$ and derive the multifractal spectrum for the pointwise Assouad dimension via their dimensions. To that end, given $\alpha\geq 0$, we let
\begin{align*}
&\PP(\alpha)=\left\{\qq\in\PP\colon \frac{H_{\pp}(\qq)}{\lambda(\qq)}\leq \alpha\right\}\,,\\
&\Gamma_n(\alpha)=\left\{\aaa\in\Gamma_n\colon \frac{H_{\pp}(\xii(\aaa))}{\lambda(\xii(\aaa))}\leq \alpha\right\},\\
&\PP_{\Gamma_n}(\alpha)=\PP_{\Gamma_n(\alpha)}=\PP(\alpha)\cap \PP_{\Gamma_n}.
\end{align*}
The next lemma follows by elementary arguments from \ref{it:abundant-2} and the continuity of $\qq\mapsto \frac{H_{\pp}(\qq)}{\lambda(\qq)}$.
\begin{lemma}\label{lemma:density}
  Let $\alpha>\alpha_{\min}$. Then for all small enough $\delta>0$, $\PP_{\Gamma_n}(\alpha)$ is $\delta$-dense in $\PP(\alpha)$ for all large enough $n\in\N$.
\end{lemma}
The purpose of introducing the sets $\Gamma_n(\alpha)$ is the following: For every $n\in\N$, the symbolic space $\Sigma(\Gamma_n(\alpha))$ is a subset of $U_\alpha^\A$, since the construction essentially forces $\frac{H_{\pp}(\xii(\aaa))}{\lambda(\xii(\aaa))}\leq \alpha$ uniformly over all subwords of a given word $\iii\in\Sigma(\Gamma_n(\alpha))$. Moreover, the Hausdorff dimension of a symbolic space is easy to calculate and the abundance of the sequence $\Gamma_n$ together with some combinatorial arguments allows us to approximate $\barf(\alpha)$ arbitrarily well with the Hausdorff dimensions of the sets $\Sigma(\Gamma_n(\alpha))$. This is made formal in the following lemmata. Note that in what follows, we consider $\Sigma(\Gamma_n(\alpha))$ as a subset of $\Sigma=\Sigma(\Lambda)$. 

\begin{lemma}\label{lemma:subset}
  For all $\alpha>\alpha_{\min}$, all large enough $n$ and all $\iii\in\Sigma(\Gamma_n(\alpha))$, we have
  \begin{equation*}
    \dima(\nu,\iii)\leq \alpha.
  \end{equation*}
\end{lemma}
\begin{proof}
  For a fixed $\alpha>\alpha_{\min}$, the set $\Gamma_n(\alpha)$ is non empty for all large enough $n$ by Lemma \ref*{lemma:density}. Let $\iii\in \Sigma(\Gamma_n(\alpha))$ and let $\aaa$ be a finite subword of $\iii$, with $|\aaa|\geq 2n$. It follows that
  \begin{equation*}
      \aaa=\bbb\aaa_1\aaa_2\ldots\aaa_k\bbb',
  \end{equation*}
  where $\aaa_\ell\in \Gamma_n(\alpha)$, for all $\ell=1,\ldots,k$ and $|\bbb|,|\bbb'|\leq n$, and thus
  \begin{align*}
      \frac{H_{\pp}(\xii(\aaa))}{\lambda(\xii(\aaa))}&=\frac{\sum_{i\in\Lambda} N_i(\aaa)\log p_i}{\sum_{i\in\Lambda} N_i(\aaa)\log r_i}\\
      &\leq\frac{\sum_{\ell=1}^k\sum_{i\in\Lambda} N_i(\aaa_\ell)\log p_i+2n\log p_{\min}}{\sum_{\ell=1}^k\sum_{i\in\Lambda} N_i(\aaa_\ell)\log r_i}
      \leq\alpha+\frac{2 \log p_{\min}}{k\log r_{\max}}\,.        
  \end{align*}
  Letting $|\aaa|\to\infty$, sends $k\to\infty$ and recalling Lemma \ref{lemma:explicit-assouad} gives the claim.
\end{proof}
The following technical lemma is used to bound the Hausdorff dimension of $\Sigma(\Gamma_n(\alpha))$ from below.
\begin{lemma}\label{lemma:P_n-density}
  Let $\alpha\in (\alpha_{\min},\alpha_{\max}]$ and let $(\Gamma_n)_n$ be an abundant sequence. Then for any $\varepsilon>0$, and for all large enough $n\in\N$, there exists $\qq\in \PP_{\Gamma_n}(\alpha)$, such that
  \begin{equation*}
      \frac{H(\qq)}{\lambda(\qq)}\geq \overline{f}(\alpha)-\varepsilon.
  \end{equation*}
\end{lemma}
\begin{proof}
  Assume first that $\alpha_{\min}<\alpha=\alpha(q)\le\alpha(0)$. Let $\varepsilon>0$, $\tau=\tau(q)$ and $\oomega=(p_1^qr_1^{\tau},\ldots,p_m^q r^{\tau}_m)\in\PP$ and note that 
  \begin{equation*}
  \frac{H_{\pp}(\oomega)}{\lambda(\oomega)}=\frac{\sum_{i=1}^mp_i^qr_i^{\tau}\log p_i}{\sum_{i=1}^mp_i^qr_i^{\tau}\log r_i}=\alpha\,.
  \end{equation*}
  Observe that the map $\qq\mapsto \frac{H(\qq)}{\lambda(\qq)}$ is continuous and recall that by Lemma \ref{lemma:density}, $\PP_{\Gamma_n}(\alpha)$ is $\delta$-dense in $\PP(\alpha)$ for all large $n$. Choosing $\delta$ small and $n\in\N$ large, we may thus pick $\qq\in\PP_{\Gamma_n}(\alpha)$, such that
  \begin{equation*}
  \frac{H(\qq)}{\lambda(\qq)}\geq \frac{H(\oomega)}{\lambda(\oomega)}-\varepsilon=\frac{\sum_{i=1}^mp_i^qr_i^{\tau}\log p_i^qr_i^{\tau}}{\sum_{i=1}^mp_i^qr_i^{\tau}\log r_i}-\varepsilon=\overline{f}(\alpha)-\varepsilon\,,
  \end{equation*}
  where, in the last equality, we used \eqref{eq:mf_expl} and the fact that $f(\alpha)=\overline{f}(\alpha)$ since we have assumed that $\alpha\le\alpha(0)$.
  
  If $\alpha>\alpha(0)$, we let $\oomega=(r_1^s,\ldots,r_m^s)$ where $s=\dimh\Sigma=\tau(0)=\barf(\alpha)$. Then
  \begin{equation*}
  \frac{H_{\pp}(\oomega)}{\lambda(\oomega)}=\frac{\sum_{i=1}^mr_i^{s}\log p_i}{\sum_{i=1}^mr_i^{s}\log r_i}=\alpha(0)<\alpha\,,
  \end{equation*}
  so $\oomega\in\PP(\oomega)$. Arguing as above, there is $\qq\in\PP_{\Gamma_n}(\alpha)$, such that
  \begin{equation*}
  \frac{H(\qq)}{\lambda(\qq)}\geq \frac{H(\oomega)}{\lambda(\oomega)}-\varepsilon=s-\varepsilon=\overline{f}(\alpha)-\varepsilon\,,
  \end{equation*}
  as required.
  \end{proof}

\begin{lemma}\label{lemma:sub-level-dimension}
  Let $\alpha\in(\alpha_{\min},\alpha_{\max}]$, and let $(\Gamma_n)_n$ be an abundant sequence. Then for any $\varepsilon>0$, there is $n\in\N$, such that
  \begin{equation*}
      \dimh \Sigma(\Gamma_n(\alpha))\geq\overline{f}(\alpha)-\varepsilon.
  \end{equation*}
\end{lemma}
\begin{proof}
  Let $\alpha\in(\alpha_{\min},\alpha_{\max}]$ and let $\varepsilon>0$. By Lemma \ref{lemma:density}, the symbolic space $\Sigma(\Gamma_n(\alpha))$ is non-empty for all large $n$. It is well known that the Hausdorff dimension of $\Sigma(\Gamma_n(\alpha))$ is the unique $s_n$ satisfying
  \begin{equation}\label{eq:sigma-n-dim}
      \sum_{\aaa\in\Gamma_n(\alpha)}r_{\aaa}^{s_n}=1.
  \end{equation}
  Using Lemma \ref{lemma:P_n-density}, for all large enough $n\in\N$, we may pick $\qq_n\in \PP_{\Gamma_n}(\alpha)$ satisfying
  \begin{equation}\label{eq:H-lambda-lim}
    \frac{H(\qq_n)}{\lambda(\qq_n)}\geq \overline{f}(\alpha)-\frac{\varepsilon}{2}.
  \end{equation} 
  We note that $\xii(\aaa)=\qq_n$ for all $\aaa\in T_{\Gamma_n}(\qq_n)$ and that $r_n=r_\aaa$ is also independent of $\aaa\in T_{\Gamma_n}(\qq_n)$ and satisfies
  \begin{equation*}
  -\log r_n=n\lambda(\qq_n)\,.
  \end{equation*}
  We may thus estimate
  \begin{equation*} 1=\sum_{\aaa\in\Gamma_n(\alpha)}r_{\aaa}^{s_n}=\sum_{\qq\in\PP_n(\alpha)}\sum_{\aaa\in T_{\Gamma_n}(\qq)}r_{\aaa}^{s_n}\geq \#T_{\Gamma_n}(\qq_n)r_{n}^{s_n}.
  \end{equation*}
 Defining $t_n$ such that $\#T_{\Gamma_n}(\qq_n)r_n^{t_n}=1$, we obtain a lower bound for $s_n$. By an application of Lemma \ref{lemma:T_n-cardinality-variant},
 we have 
  \begin{align*}
     \dimh\Sigma(\Gamma_n(\alpha)) \geq t_n = \frac{\log \#T_{\Gamma_n}(\qq_n)}{-\log r_{n}}\geq
      \frac{H(\qq_n)}{\lambda(\qq_n)}+\frac{C\log (n+1)^{m-1}}{n\log r_{\max}}\,.
  \end{align*}
By choosing $n\in\N$ large enough, such that $-\frac{\log C(n+1)^{m-1}}{n\log r_{\max}}\leq \frac{\varepsilon}{2}$, and using \eqref{eq:H-lambda-lim} we arrive at the claim. 
\end{proof}

\subsection{Construction of large Moran subsets}
In this section, we derive Theorem \ref{thm:symbolic-multifractal-spectrum} by constructing large Moran subsets inside the level sets $D_{\alpha}^\A$. The following Hausdorff dimension formula is well known (see e.g. \cite[Proposition 3.1]{FengLauWu2002} or \cite[Theorem 4.6]{KaenmakiLiSuomala2016}).
\begin{lemma}\label{prop:moran-dimension}
Let $n_k\in\N$ for all $k\in\N$ and let $\Delta_k\subset\Sigma_{n_k}$. If $n_k\le C k$ for some constant $C<\infty$, then
Then 
\[\dimh(\Delta_1\times\Delta_2\times\cdots)=\liminf_{k\to\infty} s_k\,,\]
  where $s_k$ is the unique number satisfying the equation $\sum_{\iii\in\Delta_1\times\cdots\times\Delta_k}r_{\iii}^{s_k}=1$.    
\end{lemma}

Our aim is now, for every $\alpha\in(\alpha_{\min},\alpha_{\max}]$ and $\varepsilon>0$, to construct a Moran subset $\Omega_\varepsilon(\alpha)\subset D_{\alpha}^\A$, which satisfies $\dimh \Omega_\varepsilon(\alpha)\geq \barf(\alpha)-\varepsilon$. Going forward, for a given $\Gamma\subset \Sigma_n$, we let $\alpha_{\min}(\Gamma)=\min_{\aaa\in\Gamma}\frac{\log p_\aaa}{\log r_\aaa}$ and $\alpha_{\max}(\Gamma)=\max_{\aaa\in\Gamma}\frac{\log p_\aaa}{\log r_\aaa}$. The following simple lemma will be useful in ensuring that the Moran subsets can be constructed via abundant sequences.
\begin{lemma}\label{lemma:exists-point-with-dim}
  Let $\nu$ be the Bernoulli measure associated to the probability vector $\pp$, $n\in\N$, $\Gamma\subset\Sigma_n$ and $\alpha\in [\alpha_{\min}(\Gamma),\alpha_{\max}(\Gamma)]$. Then, there exists $\iii\in\Sigma(\Gamma)$, such that
  \begin{equation*}
    \dima(\nu,\iii)=\alpha.
  \end{equation*}
\end{lemma}
\begin{proof}
  To ease notation, we denote by $\aaa_{\min}$ the symbol in $\Gamma$ minimising the ratio $\frac{\log p_{\aaa}}{\log r_{\aaa}}$ and by $\aaa_{\max}$ the symbol maximising the same ratio. Then $\frac{\log p_{\aaa_{\min}}}{\log r_{\aaa_{\min}}}\le \alpha\le\frac{\log p_{\aaa_{\max}}}{\log r_{\aaa_{\max}}}$.
  We define $\iii$ inductively by setting $i_0={\aaa_{\max}}$ and
  \begin{equation*}
      i_{n+1}=\begin{cases}
          \aaa_{\max}&\text{ if }\frac{\log p_{\iii|_n}}{\log r_{\iii|_n}}<\alpha\,,\\
          \aaa_{\min}&\text{ if }\frac{\log p_{\iii|_n}}{\log r_{\iii|_n}}\ge\alpha\,.
      \end{cases}
  \end{equation*}
  Using Lemma \ref{lemma:explicit-assouad}, it is easy to check that $\dima(\nu,\iii)=\alpha$.  
  \end{proof}
Let $\varepsilon>0$ and $\alpha\in(\alpha_{\min},\alpha_{\max})$. We start by picking $n\in\N$ so that $\alpha_{\min}(\Gamma_n)<\alpha<\alpha_{\max}(\Gamma_n)$. Note that this is possible by \ref{it:abundant-2} since the value $\frac{\log p_\aaa}{\log r_\aaa}$ is a continuous function of the type of $\aaa$. Using Lemma \ref{lemma:exists-point-with-dim}, we then pick $\iii\in\Sigma(\Gamma_n)$ such that $\dim_A(\nu,\iii)=\alpha$. By making $n$ larger if necessary, by Lemma \ref{lemma:sub-level-dimension} we may assume that
\begin{equation*}
  \dimh \Sigma(\Gamma_n(\alpha))\geq\barf(\alpha)-\frac{\varepsilon}{2}.
\end{equation*}
Denote by $s\coloneqq\barf(\alpha)-\varepsilon<\dimh \Sigma(\Gamma_n(\alpha))$, and note that
\begin{equation*}
  \sum_{\aaa\in\Gamma_n(\alpha)}r_{\aaa}^s>1,
\end{equation*}
by \eqref{eq:sigma-n-dim}. Since $r_{\iii|_k}\ge c^k$ for some constant $c>0$, we may choose $M_k\le C k\in\N$ such that
\begin{equation}\label{eq:Mk}
  \Bigg(\sum_{\aaa\in\Gamma_n(\alpha)}r_{\aaa}^s\Bigg)^{M_k} r_{\iii|_k}^s>1.
\end{equation}
Now we define
\begin{equation}\label{eq:Omega-construction}
  \Omega_\varepsilon(\alpha)= \Gamma_n(\alpha)^{M_1}\times\{\iii|_n\}\times \Gamma_n(\alpha)^{M_2}\times\{\iii|_{2n}\}\times \Gamma_n(\alpha)^{M_3}\times\{\iii|_{3n}\}\times\cdots\,.
\end{equation}
and note that $\Omega_\varepsilon(\alpha)\subset\Sigma(\Gamma_n)$ since $\iii\in\Sigma(\Gamma_n)$ and $\Gamma_n(\alpha)\subset\Gamma_n$.
The next lemma ensures that $\Omega_\varepsilon(\alpha)$ has the properties we are after.

\begin{lemma}\label{lemma:Omega-properties}
  For all $\alpha\in(\alpha_{\min},\alpha_{\max})$ and $\varepsilon>0$, the set $\Omega_\varepsilon(\alpha)$ satisfies
  \begin{enumerate}
    \item $\Omega_\varepsilon(\alpha)\subset D_\alpha^\A$,
    \item $\dimh\Omega_\varepsilon(\alpha)\geq\barf(\alpha)-\varepsilon$.
  \end{enumerate}
  
\end{lemma}
\begin{proof}
  First we show that $\Omega_n(\alpha)\subset D_\alpha^\A$. Let $\iii$ be as defined above. If $\jjj\in\Omega_n(\alpha)$, then it follows from Lemma \ref{lemma:explicit-assouad}, that $\dima(\nu,\jjj)\geq \alpha$, since for every $\aaa\sqsubset\iii$, we have $\aaa\sqsubset\jjj$ by construction, and since $\dima(\nu,\iii)=\alpha$. The upper bound $\dima(\nu,\jjj)\le\alpha$ also follows similarly to Lemma \ref{lemma:subset} since for any $\bbb\sqsubset\jjj$ the word $\bbb$ is 'almost' in $\Gamma_n(\alpha)$: It is a finite concatenation of longer and longer words in $\Gamma_n(\alpha)$ and of the subwords of $\iii$, where $\dima(\nu,\iii)=\alpha$. This completes the proof of the first claim.

  We derive the second claim using Lemma \ref{prop:moran-dimension}.
  The Hausdorff dimension of $\Omega_n(\alpha)$ is given by $\liminf t_k$ where $t_k$ is the unique number satisfying
  \begin{equation}\label{eq:sk}
\prod_{j=1}^{k}\sum_{\aaa\in\Gamma_n(\alpha)^{M_j}}r_{\aaa}^{t_k}r_{\iii|_j}^{t_k}=1.
  \end{equation}
    By \eqref{eq:Mk}, for every $k\in\N$, we have
  \begin{equation*}
\prod_{j=1}^{k}\sum_{\aaa\in\Gamma_n(\alpha)^{M_j}}r_{\aaa}^{s}r_{\iii|_j}^{s}=\prod_{j=1}^{k}\Bigg(\sum_{\aaa\in\Gamma_n(\alpha)}r_{\aaa}^{s}\Bigg)^{M_j}r_{\iii|_j}^{s}>1.
  \end{equation*}
  Combining this with \eqref{eq:sk} and the observation that the map
  \begin{equation*}
    t\mapsto \prod_{j=1}^{k}\sum_{\aaa\in\Gamma_n(\alpha)^{M_j}}r_{\aaa}^{t}r_{\iii|_j}^{t}
  \end{equation*}
  is decreasing for all $k\in\N$, it follows that $t_k>s$ for all $k\in\N$, and therefore
  \begin{equation*}
    \dimh\Omega_n(\alpha)=\liminf_{k\to\infty}t_k\geq s=\barf(\alpha)-\varepsilon\,,
  \end{equation*}
  as required.
\end{proof}

To achieve the full multifractal spectrum, we need to treat the case when $\alpha=\alpha_{\min}$ separately. For this, let $\Lambda_{\min}=\{i\in\Lambda\colon \frac{\log p_i}{\log r_i}=\min_{i\in\Lambda}\frac{\log p_i}{\log r_i}\}$. The following lemma follows immediately from Lemma \ref{lemma:explicit-assouad} and Theorem \ref{thm:classical-formalism} for $\alpha=\alpha_{\min}$.
\begin{lemma}\label{lemma:alpha_min-lower-bound}
  We have $\Sigma(\Lambda_{\min})\subset D_{\alpha_{\min}}^\A$ and $\dimh\Sigma(\Lambda_{\min})=\barf(\alpha_{\min})$.
\end{lemma}

We may now conclude with the proof of Theorem \ref{thm:symbolic-multifractal-spectrum} yielding the full multifractal spectrum for the pointwise Assouad dimension of Bernoulli measures.  
\begin{proof}[Proof of Theorem \ref{thm:symbolic-multifractal-spectrum}]
  For $n\in\N$, we let $\Gamma_n=\Sigma_n$. Then $(\Gamma_n)_{n\in\N}$ trivially is an abundant sequence. Let $\alpha\in[\alpha_{\min},\alpha_{\max}]$. Since $D_\alpha^\A\subset U_{\alpha}^\A\subset U_\alpha^\loc$, the upper bound follows from Theorem \ref{thm:classical-formalism}. For $\alpha=\alpha_{\min}$, the lower bound follows from Lemma \ref{lemma:alpha_min-lower-bound} and for $\alpha\in(\alpha_{\min},\alpha_{\max})$ and $\varepsilon>0$, Lemma \ref{lemma:Omega-properties} implies that
  \begin{equation*}
    \dimh D_\alpha^\A\geq \dimh \Omega_\varepsilon(\alpha)\geq \barf(\alpha)-\varepsilon\,.
  \end{equation*}
  Finally, for $\alpha=\alpha_{\max}$, we may define $\Omega_{\varepsilon}(\alpha)$ using the word $\iii=i_{\max}i_{\max}\ldots$ in \eqref{eq:Omega-construction}, where $i_{\max}$ is the symbol in $\Lambda$ maximising $\frac{\log p_i}{\log r_i}$. Note that in this case, it is no longer necessarily true that $\Omega_\varepsilon(\alpha)\subset\Sigma(\Gamma_n)$, but nevertheless Lemma \ref{lemma:Omega-properties} goes through unchanged and the proof finishes as before. Since $\frac{\log p_{\aaa}}{\log r_{\aaa}}\in[\alpha_{\min},\alpha_{\max}]$, for any $\aaa\in\Sigma_*$, the second claim readily follows from Lemma \ref{lemma:explicit-assouad}.
\end{proof}

\section{Multifractal analysis of self-similar measures}\label{sec:OSC}
In this section, we fix a self-similar measure $\mu$ supported on a self-similar set $X\subset\R^d$ satisfying the OSC. We wish to apply the results of Section \ref{sec:symbolic} to deduce Theorem \ref{thm:main-spectrum}. The absence of geometric separation of cylinders provides additional difficulties compared to the symbolic case. Heuristically, this is because the codings of geometrically adjacent cylinders may be very far apart in the symbolic space, which may result in nearby cylinders having incomparable masses. This is also the reason why self-similar measures with OSC may fail to be doubling.

Let $U$ be the open set witnessing the OSC. By a theorem of Schief \cite{Schief1994}, we may assume that $X\cap U\ne\emptyset$ and therefore, we may fix a finite word $\kkk\in\Sigma_*$ satisfying $\varphi_{\kkk}(X)\subset U$. We also recall the following characterisation of the OSC \cite{Schief1994}: There is a constant $M<\infty$ such that $\#\Sigma(x,r)\leq M$, for all $x\in X$ and $r>0$, where
\begin{equation*}
  \Sigma(x,r)=\{\iii\in\Sigma_*\colon r_{\iii}\leq r < r_{\iii^-},\,\varphi_{\iii}(X)\cap B(x,r)\ne\emptyset\}\,.
\end{equation*}
From now on, we fix such a constant $M$.

Appending $\kkk$ to any word $\iii\in\Sigma_*$ allows us to control the measures of geometric balls with their symbolic counterparts.
\begin{lemma}\label{lemma:osc-ball-measure}
    There are constants $C,\delta>0$, such that
    \begin{equation*}
        Cp_{\iii}\leq\mu(B(x,\delta r_{\iii}))\leq p_{\iii},
    \end{equation*}
    for any $\iii\in\Sigma_*$ and $x\in \varphi_{\iii\kkk}(X)$.
\end{lemma}
\begin{proof}
Let $\iii\in\Sigma_*$ and $x\in\varphi_{\iii}(X)$ and observe first that the lower bound is trivial, since $\varphi_{\iii i}(X)\subset B(x,r_{\iii})$, for any $i\in\Lambda$. It remains to prove the upper bound. Since $\varphi_{\kkk}(X)$ is compact and $\varphi_{\kkk}(X)\subset U$, we have that $\dist(\varphi_{\kkk}(X),\partial U)>0$ and therefore we may choose $0<\delta<\dist(\varphi_{\kkk}(X),\partial U)$. By self-similarity, we then have
\begin{equation}\label{eq:delta-distance}
    \dist(\varphi_{\iii\kkk}(X),\varphi_{\iii}(\partial U))>\delta r_{\iii}.
\end{equation}
In particular, since the OSC implies that $\varphi_{\iii}(U)\cap \varphi_{\jjj}(U)=\emptyset$ for all $\jjj\in \Sigma_{|\iii|}$ with $\jjj\ne\iii$ and since $\varphi_{\iii\kkk}(X)\subset \varphi_{\iii}(U)$, we get by \eqref{eq:delta-distance} that for any $x\in \varphi_{\iii\kkk}(X)$, we have $B(x,\delta r_{\iii})\cap X\subset \varphi_{\iii}(X)$ and therefore
\begin{equation*}
    \mu(B(x,\delta r_{\iii}))\leq p_{\iii}.
\end{equation*}
\end{proof}
Towards a proof of Theorem \ref{thm:main-spectrum}, we begin with the following proposition, which generalises Theorem \ref{thm:exact-dim} by dropping the assumption that $\mu$ is doubling.

\begin{proposition}
    \label{prop:exact-dim-non-doubling}
  Let $\mu$ be a self-similar measure satisfying the OSC. Then for any ergodic measure $\nu$ on $X$, we have
  \begin{equation*}
    \dima(\mu,x)=\dima\mu,
  \end{equation*}
  for $\nu$-almost every $x\in X$.
  \end{proposition}
\begin{proof}
  If $\mu$ is doubling, the claim is precisely Theorem \ref{thm:exact-dim}. Therefore, we assume that $\mu$ is not doubling, so that $\dima\mu=\infty$. It then suffices to show that at $\nu$-almost all points, $\mu$ is not pointwise doubling. For $A\subset X$ and $\varepsilon>0$ we let $U_\varepsilon(A)$ denote the open $\varepsilon$-neighbourhood of $A$. Since $\mu$ is not doubling, by \cite[Theorem 2.3]{Yung2007}, for any $n\in\N$, there are words $\iii_n,\jjj_n\in\Sigma_*$ satisfying $\varphi_{\jjj_n}(X)\subset U_{r_{\iii_n}}(\varphi_{\iii_n}(X))$ and $p_{\jjj_n}\geq np_{\iii_n}$. 
  We let
  \begin{equation*}
    \mathcal{N}=\{\iii\in\Sigma\colon\iii_n\kkk\sqsubset\iii,\,\forall n\in\N\}
  \end{equation*}
  and observe that a simple application of Birkhoff's ergodic theorem gives $\nu(\mathcal{N})=1$, see \cite[Lemma 4.7]{Anttila2022}. We will finish the proof by showing that if $\iii\in \mathcal{N}$, then $\mu$ is not doubling at $\pi(\iii)$.
  To that end, let  $\iii\in\mathcal{N}$ and for $n\in\N$ choose $m\in\N$, such that $\iii|_{m+|\kkk|}=\jjj\iii_n\kkk$, where $\jjj=\iii|_{n-|\iii_n\kkk|}$, and let $r_n=r_{\iii|_m}$. By Lemma \ref{lemma:osc-ball-measure}, we have
    \begin{equation*}
        \mu(B(\pi(\iii),\delta r_n))\leq p_{\iii_m}=p_{\jjj}p_{\iii_n}.
    \end{equation*}
  Moreover, as $\varphi_{\jjj \jjj_n}(X)\subset B(\varphi_{\iii|_m}(X),r_{n})\subset B(\pi(\iii),2r_{n})$, it follows that
  \begin{equation*}
    \mu(B(\pi(\iii),2r_n))\geq p_{\jjj\jjj_n}=p_{\jjj}p_{\jjj_n}\geq n p_{\jjj}p_{\iii_n}\geq n\mu(B(\pi(\iii),\delta r_n)).
  \end{equation*}
  Since this holds true for all $n\in\N$, it follows that $\mu$ is not doubling at $\pi(\iii)$.
\end{proof}
The previous proposition implies Claim \eqref{thm:main-2} in Theorem \ref{thm:main-doubling}.
\begin{corollary}\label{cor:main-2}
  Let $\mu$ is a non-doubling self-similar measure satisfying the OSC and let $s=\dimh X$. Then
  \begin{equation*}
    \HH^{s}(D(\mu))=0.
  \end{equation*}
\end{corollary}
\begin{proof}
  Recall that $\HH^{s}$ is equivalent to the self-similar measure $\nu$ corresponding to the probability vector $r_1^{s},\ldots,r_m^s$. Since all self-similar measures are ergodic, the claim follows from Proposition \ref{prop:exact-dim-non-doubling}.
\end{proof}

In light of Proposition \ref{prop:exact-dim-non-doubling}, there is no a priori expectation for the sets $D_\alpha^\A$ to be empty for $\alpha_{\max}<\alpha<\infty$. However, as will be shown soon, this in fact is the case. Let us begin with an illustrative example

\begin{example}\label{ex:Ber}
Let $\mu$ be the image of the $(p,1-p)$-Bernoulli measure on the unit interval corresponding to the self similar IFS $\{x\mapsto x/2,x\mapsto1/2+x/2\}$. Suppose $p<\tfrac12$. Then $\mu$ is not doubling, $\alpha_{\min}=\tfrac{\log(1-p)}{-\log 2}$ and $\alpha_{\max}=\tfrac{\log p}{-\log2}$. It is easy to see that $\dima(\mu,0)=\alpha_{\max}$ and that $\dima(\mu,q)=\alpha_{\min}$ for all dyadic rationals $q\in(0,1]$. Suppose $x\in[0,1]$ is not a dyadic rational, let $\iii\in\{0,1\}$ be the binary coding of $x$ and assume that $\mu$ is doubling at $x$. Then it is readily verified that the number of consecutive zero's in $\iii$ is bounded by a constant, say $N$. Consider $0<r<R<1$ and let $k,m$ be integers satisfying $2^{-k}>2R\ge 2^{-k-1}$ and $2^{-k-m+1}>r\ge 2^{-k-m}$. Let $J$ be the dyadic interval of length $2^{-k-m}$ containing $x$. If  $B(x,R)$ is contained in a unique dyadic interval $I$ of length $2^{-k}$, we have
\begin{align*}
   \mu(B(x,R))&\le\mu (I)\le p^{-m}\mu(J)\le p^{-m}\mu(B(x,r))=2^{m\alpha_{\max}}\mu (B(x,r))\\&\le 2^{3\alpha_{\max}}\mu(B(x,r))\left(\frac{R}{r}\right)^{\alpha_{\max}}. 
\end{align*}
On the other hand, if $B(x,R)$ intersects two adjacent dyadic intervals $I_1,I_2$ of length $2^{-k}$, let $a$ denote their common endpoint, say the left endpoint of $I_2$ and the right endpoint of $I_1$. Then $\mu(I_1)\ge\tfrac{p}{1-p}\mu(I_2)$, so if $x\in I_1$, then 
\begin{align*}
    \mu(B(x,R))&\le p^{-1}\mu(I_1)\le p^{-m-1}\mu(J)\le p^{-m-1}\mu(B(x,r))\\
    &\le 2^{4\alpha_{\max}}\mu(B(x,r))\left(\frac{R}{r}\right)^{\alpha_{\max}}\,.
\end{align*}
Finally, if $x\in I_2$, it follows that
\[\mu(I_1)\le\left(\frac{1-p}{p}\right)^{N-1}\mu(I_2)\,,\]
so that
\begin{align*}
    \mu(B(x,R))&\le \mu(I_1)+\mu(I_2)\\
    &\le\left(1+\left(\frac{1-p}{p}\right)^{N-1}\right)\mu(I_2)\\
    &\le \left(1+\left(\frac{1-p}{p}\right)^{N-1}\right)p^{-m}\mu(J)\\
    &\le C\mu(B(x,r))\left(\frac{R}r\right)^{\alpha_{\max}},
\end{align*}
for $C=2^{3\alpha_{\max}}(1+(\tfrac{1-p}{p})^{N-1})$. It is easy to see that also $\dima(\mu,x)\ge\alpha_{\min}$ for all $x\in[0,1]$. So we conclude that for all $x$, either $\dima(\mu,x)=\infty$ (if $\mu$ is not doubling at $x$) or else $\alpha_{\min}\le\dima(\mu,x)\le\alpha_{\max}$.
\end{example}

The phenomenon of the previous example generalises to self-similar measures on $\R^d$ satisfying the OSC. In the general case the proof is more complicated partly since, unlike in Example \ref{ex:Ber}, it is difficult to identify the codings of geometrically adjacent cylinders. However, if the measure is doubling at $x$, this essentially allows us to transfer mass between adjacent cylinders at exponentially separated scales. For such points $x$, this implies that the measure of balls $B(x,r)$ scales as in the symbolic setting up to an arbitrarily small error in the exponent. This is the content of Lemma \ref{lemma:case-study} which amounts to the main technical part of the proof of the following theorem.

\begin{theorem}\label{thm:doubling-iff-leq-alpha_max}
  Let $\mu$ be a self-similar measure satisfying the OSC. Then
  \begin{equation*}
    \alpha_{\min}\leq\dima(\mu,x)\leq \alpha_{\max}
  \end{equation*}
  if and only if $\mu$ is doubling at $x$. In particular, $D_\alpha^\A=\emptyset$ for $\alpha\not\in[\alpha_{\min},\alpha_{\max}]\cup\{\dima\mu\}$
\end{theorem}
\begin{proof}
If $\alpha_{\min}\leq\dima(\mu,x)\leq \alpha_{\max}<\infty$, then $\mu$ is doubling at $x$ by \cite[Proposition 3.1]{Anttila2022}. Towards the converse implication, our first goal is to first verify the estimate $\dima(\mu,x)\leq \alpha_{\max}$. The following geometric lemma is the key technical ingredient of the argument; See Figure \ref{fig:lemma-fig} for an illustration.
\begin{figure}
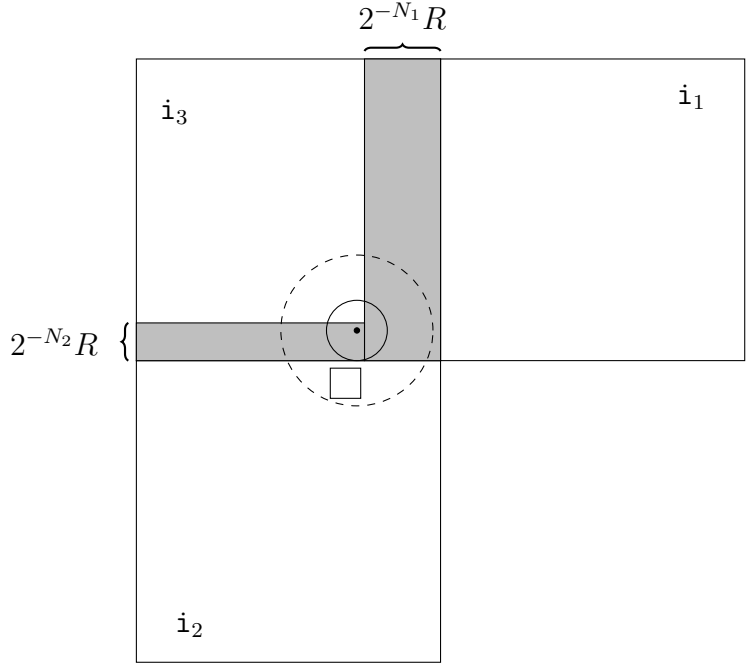

    \captionsetup{justification=justified}
    \includestandalone{pw-lemma-case-1}
    \caption{Proof of Lemma \ref{lemma:case-study}: The point in the picture is $x\in A_2$ and letting $d=d(x,\varphi_{\iii_2}(X))$, we see that $B(x,2d)$ contains a sub-cylinder of $[\iii_2]$ of size comparable to $d$. This gives a lower bound for $\mu(B(x,d))$. The choice of $N_2$ ensures that a sufficiently large portion of the mass of $B(x,R)$ comes from $[\iii_2]$.}
    \label{fig:lemma-fig}
\end{figure}
\begin{lemma}\label{lemma:case-study}
  Let $\mu$ be a self-similar measure satisfying the OSC, and assume that $\mu$ is doubling at $x$. Then for all $\varepsilon>0$ there exists $N\in\N$ and a constant $C_{\varepsilon}=C(\varepsilon,x,\mu)>0$, such that for all $0<r<R$, either
  \begin{equation}\label{eq:case-1}
    \frac{\mu(B(x,R))}{\mu(B(x,r))}\leq C_{\varepsilon}\left(\frac{R}{r}\right)^{\alpha_{\max}},
  \end{equation}
  or there exists a scale $r<d\leq2^{-N}R$, such that
  \begin{equation}\label{eq:case-2}
    \frac{\mu(B(x,R))}{\mu(B(x,d))}\leq \left(\frac{R}{d}\right)^{\alpha_{\max}+\varepsilon}.
  \end{equation}
\end{lemma}
\begin{proof}
  For a given $\varepsilon>0$, we let $N_0=1$ and $N=N_1\in \N$ be large enough to satisfy
  \begin{equation*}
    \frac{N_0\log C}{N_1\log 2}<\frac{\varepsilon}{2}\quad\text {and}\quad\frac{CM}{p_{\min}r^{\alpha_{\max}}_{\min}}\leq 2^{\varepsilon N_1/2}\,,
  \end{equation*}
  where $C>1$ is the doubling constant of $\mu$ at $x$.
  Moreover, we inductively choose integers $N_2\leq N_3\leq \ldots \leq N_M$ satisfying
  \begin{equation*}
    \frac{N_{i-1}\log C}{N_i\log 2}<\frac{\varepsilon}{2}.
  \end{equation*}
  Applying the doubling property repeatedly, we get
  \begin{equation}\label{eq:N_i-bound}
    \mu(B(x,R))\leq C^{N_i}\mu(B(x,2^{-N_i}R))\le C^{N_M}\mu(B(x,2^{-N_i}R))\,,
  \end{equation}
  for all $i=1,\ldots, M$ and all $R>0$.

  Let us now pick $0<r<R$ and enumerate $\Sigma(x,R)=\{\iii_i\}_{i=1}^\ell$ in the order of decreasing mass, that is so that $p_{\iii_1}\geq p_{\iii_2}\geq \ldots \geq p_{\iii_\ell}$. Recall that $\ell\leq M$.
  Let $A_1=\{y\in X\colon d(y,\varphi_{\iii_1}(X))\leq 2^{-M_1}R\}$, and inductively,
  \begin{equation*}
    A_i=\{y\in X\colon d(y,\varphi_{\iii_i}(X))\leq 2^{-N_i}R,\,d(y,\varphi_{\iii_{j}}(X))> 2^{-N_{j}}R, \text{ for all }j<i\},
  \end{equation*}
  for each $i=2,\ldots,\ell$. Then it is easy to see that $x\in \bigcup_{i=1}^\ell A_i$. Indeed, at least one of the words $\{\iii_i\}_{i=1}^\ell$, say $\iii_j$, is a prefix of a coding of $x$ and therefore $d(x,\varphi_{\iii_j}(X))=0\leq 2^{-N_j}R$. Now either $x\in \bigcup_{i=1}^{j-1} A_i$, or 
  we have that $d(x,\varphi_{\iii_i}(X))>2^{-N_i}R$, for all $i\leq j-1$, so $x\in A_j$.
  
  Now, we fix $1\leq i \leq \ell$, such that $x\in A_i$. Since $N_1\leq N_2\leq\ldots\leq N_M$, we have $d(x,\varphi_{\iii_{j}}(X))>2^{-N_j}R\geq 2^{-N_{i-1}}R$, for all $j<i$. In particular, $B(x,2^{-N_{i-1}}R)\cap X\subset\bigcup_{k=i}^\ell\varphi_{\iii_k}(X)$, and thus
  \begin{equation}\label{eq:N_i-upper-bound}
    \mu(B(x,2^{-N_{i-1}}R))\leq\sum_{k=i}^\ell p_{\iii_k}\leq M p_{\iii_i}. 
  \end{equation}
  Let us divide the proof into two cases:

  \subsection*{Case 1:}
  If $d(x,\varphi_{\iii_i}(X))\leq r$, then there is a point $y\in \varphi_{\iii_i}(X)$, such that $y\in B(x,r)$. Recall that $\varphi_{\iii_i}(X)$ is compact. Let $\jjj$ be a coding of $y$. Since $y\in \varphi_{\iii_i}(X)$, we may assume that $\jjj|_n=\iii_i$, where $n\in\N$ is the unique integer satisfying $r_{\jjj|_n}\leq R < r_{\jjj|_{n-1}}$. Then we choose $k\in\N$ to be the unique integer satisfying $r_{\jjj|_{n+k}}\leq r < r_{\jjj|_{n+k-1}}$. Since $r > r_{\jjj|_{n+k+1}}$, we have $\varphi_{\jjj_{n+k+1}}(X)\subset B(y,r)$ and therefore,
  \begin{equation*}
    \mu(B(y,r))\geq p_{\min}p_{\jjj|_{n+k}}\,.
  \end{equation*}
  Recalling that $C$ is the doubling constant of $\mu$ at $x$ and by using the fact that $d(x,\varphi_{\iii_i}(X))\leq r$, we have
  \begin{equation}\label{eq:alaraja}
    \mu(B(x,r))\geq C^{-1}\mu(B(x,2r))\geq C^{-1}\mu(B(y,r))\geq C^{-1}p_{\min}p_{\jjj|_{n+k}}\,,
  \end{equation}
  where $C>1$ does not depend on $\varepsilon$ nor $N$.
  Therefore, using \eqref{eq:N_i-bound} and \eqref{eq:N_i-upper-bound}
  \begin{equation}
  \begin{split}
    \frac{\mu(B(x,R))}{\mu(B(x,r))}&\leq  C^{N_M}\frac{\mu(B(x,2^{-N_{i-1}}R))}{\mu(B(x,r))}\leq C^{N_M}\frac{Mp_{\jjj_{n}}}{C^{-1}p_{\min}p_{\jjj_{n+k}}}\\
    &=C^{N_M+1}\frac{M}{p_{\min}}\left(\prod_{j=n+1}^kp_{\jjj_j}\right)^{-1}=C^{N_M+1}\frac{M}{p_{\min}}\left(\prod_{j=n+1}^kr_j^{\frac{\log p_{\jjj_j}}{\log r_{\jjj_j}}}\right)^{-1}\\
    &\leq C^{N_M+1}\frac{M}{p_{\min}}\left(\prod_{j=n+1}^kr_{\jjj_j}^{\alpha_{\max}}\right)^{-1}=C^{N_M+1}\frac{M}{p_{\min}}\left(\frac{r_{\jjj|_n}}{r_{\jjj|_{n+k}}}\right)^{\alpha_{\max}} \\
    &\leq C^{N_M+1}\frac{M}{p_{\min}r^{\alpha_{\max}}_{\min}}\left(\frac{R}{r}\right)^{\alpha_{\max}},\label{eq:upper-bound-calc}
\end{split}
\end{equation}
  so in this case, the first inequality holds with $C_{\varepsilon}=C^{N_M+1}\frac{M}{p_{\min}r^{\alpha_{\max}}_{\min}}$.
  
  \subsection*{Case 2:}
  If $r<d(x,\varphi_{\iii_i}(X))\leq 2^{-N_i}R$, then we let $d=d(x,\varphi_{\iii_i}(X))$ so that $2^{N_i}\leq \frac{R}{d}$. Since $\varphi_{\iii_i}(X)$ is compact, there is a point $y\in B(x,d)\cap \varphi_{\iii_i}(X)$ and replacing the ball $B(x,r)$ by  $B(x,d)$ in the computation \eqref{eq:alaraja}, we see that
  \begin{equation*}
    \mu(B(x,d))\geq C^{-1}p_{\min}p_{\jjj|_{n+k}}\,,
  \end{equation*}
  where $\jjj$ is a coding of $y$, satisfying $\jjj|_n=\iii_i$ and $k\in\N$ is the unique integer satisfying $r_{\jjj|_{n+k}}\leq d < r_{\jjj|_{n+k-1}}$. Using \eqref{eq:N_i-bound}, \eqref{eq:N_i-upper-bound}, and the choice of $N_i$, and computing as in \eqref{eq:upper-bound-calc}, we obtain
  \begin{align*}
    \frac{\mu(B(x,R))}{\mu(B(x,d))}&\leq C^{N_{i-1}}\frac{\mu(B(x,2^{-N_{i-1}}R))}{\mu(B(x,d))}\leq C^{N_{i-1}} \frac{Mp_{\jjj_{n}}}{C^{-1}p_{\min}p_{\jjj_{n+k}}}\\
    &\leq (2^{N_i})^{\frac{N_{i-1}\log C}{N_{i}\log 2}} \frac{CM}{p_{\min}r_{\min}^{\alpha_{\max}}}\left(\frac{R}{d}\right)^{\alpha_{\max}}\leq \left(\frac{R}{d}\right)^{\alpha_{\max}+\varepsilon}\,,
  \end{align*}
  as required.
\end{proof}

Let now $x\in X$, fix $0<r<R$ and assume that $\mu$ is doubling at $x$. Let $\varepsilon>0$ and let $N\in\N$ and $C_\varepsilon$ be as in Lemma \ref{lemma:case-study}. Let us show that, for some $k=0,1,\ldots$, there are $r<d_k\leq 2^{-N}d_{k-1}\leq\ldots\leq 2^{-kN}R=2^{-kN}d_0$ satisfying
\begin{equation}\label{eq:telescope}
  \frac{\mu(B(x,d_{m-1}))}{\mu(B(x,d_{m}))}\leq \left(\frac{d_{m-1}}{d_{m}}\right)^{\alpha_{\max}+\varepsilon},
\end{equation}
for all $m=1,\ldots,k$, and 
\begin{equation*}
  \frac{\mu(B(x,d_{k}))}{\mu(B(x,r))}\leq C_{\varepsilon}\left(\frac{d_{k}}{r}\right)^{\alpha_{\max}+\varepsilon}.
\end{equation*}

To that end, we apply Lemma \ref{lemma:case-study} for the $0<r<R$. Now either \eqref{eq:case-1} holds, in which case we are done by setting $d_0=R$, or we find $r<d_1\leq 2^{-N}R$, satisfying
\begin{equation}\label{eq:telescope-last-term}
  \frac{\mu(B(x,R))}{\mu(B(x,d_1))}\leq \left(\frac{R}{d_1}\right)^{\alpha_{\max}+\varepsilon}.
\end{equation}
In this case, we apply Lemma \ref{lemma:case-study} again for $0<r<d_1$. Again, either \eqref{eq:case-1} is satisfied with $d_1$ in place of $R$, in which case we stop the process or we find a scale $r<d_2\leq 2^{-N}d_1$, satisfying
\begin{equation*}
  \frac{\mu(B(x,d_1))}{\mu(B(x,d_2))}\leq \left(\frac{d_1}{d_2}\right)^{\alpha_{\max}+\varepsilon}.
\end{equation*}
Continuing in this manner by applying Lemma \ref{lemma:case-study} iteratively, at some point the process stops, since the sequence $d_\ell$ decreases exponentially with $\ell$. By using \eqref{eq:telescope} and \eqref{eq:telescope-last-term}, we then have
\begin{align*}
  \frac{\mu(B(x,R))}{\mu(B(x,r))}&=\frac{\mu(B(x,d_k))}{\mu(B(x,r))} \prod_{m=1}^k\frac{\mu(B(x,d_{m-1}))}{\mu(B(x,d_m))}\\
  &\leq C_{\varepsilon}\left(\frac{d_{k}}{r}\right)^{\alpha_{\max}+\varepsilon}\prod_{m=1}^{k} \left(\frac{d_{m-1}}{d_m}\right)^{\alpha_{\max}+\varepsilon}\\
  &=C_{\varepsilon}\left(\frac{R}{r}\right)^{\alpha_{\max}+\varepsilon}
\end{align*}
Since this is true for all $0<r<R$, we have $\dima(\mu,x)\leq \alpha_{\max}+\varepsilon$ and the claim follows since $\varepsilon>0$ was arbitrary.

Finally, it remains to show that $\dima(\mu,x)\geq \alpha_{\min}$. Let $0<r<\frac{R}{2}$ and let $\iii_r\in\Sigma(x,r)$ be the word satisfying $p_{\iii_r}=\max_{\iii\in\Sigma(x,r)}p_\iii$. We also let $k\in\N$ be the unique integer such that $r_{\iii_r|_k}< \frac{R}{2} \leq r_{\iii_r|_{k-1}}$. Since $\varphi_{\iii_r}(X)\subset \varphi_{\iii_r|_k}(X)$ and $\iii_r\in\Sigma(x,r)$, we have
\begin{equation*}
  d(x,\varphi_{\iii_r|_k}(X))\leq d(x,\varphi_{\iii_r}(X))\leq r < \frac{R}{2},
\end{equation*}
and for any $y\in\varphi_{\iii_r|_k}(X)$,
\begin{equation*}
  d(x,y)\leq d(x,\varphi_{\iii_r|_k}(X))+\diam(\varphi_{\iii_r|_k}(X))<\frac{R}{2}+r_{\iii_r|_k}<R.
\end{equation*}
Thus $\varphi_{\iii_r|_k}(X)\subset B(x,R)$, and a calculation similar to \eqref{eq:upper-bound-calc} gives
\begin{equation*}
  \frac{\mu(B(x,R))}{\mu(B(x,r))}\geq \frac{p_{\iii_r|_k}}{Mp_{\iii_r}}\ge\frac{1}{M} \left(\frac{r_{\iii_r|_k}}{r_{\iii_r}}\right)^{\alpha_{\min}}\geq \frac{r_{\min}^{\alpha_{\min}}}{2M} \left(\frac{R}{r}\right)^{\alpha_{\min}},
\end{equation*}
which finishes the proof.
\end{proof}

We are now ready to prove Theorem \ref{thm:main-spectrum}. We exploit the choice of the word $\kkk$, and build an abundant sequence in the symbolic space, which allows us to apply the techniques developed in Section \ref{sec:symbolic}. More precisely, for $n\geq |\kkk|$, we let
\begin{equation*}
    \Gamma_n=\{\iii\kkk\in \Sigma_{n}\colon \iii\in\Sigma_{n-|\kkk|}\}.
\end{equation*}
It is routine to verify that the sequence $(\Gamma_n)$ is abundant. Passing to the sub-self-similar set defined by using $\Gamma_n$ as the alphabet ensures enough separation to calculate the pointwise Assouad dimension on the symbolic side:

\begin{lemma}\label{lemma:subset-non-doubling}
  For any $\iii\in\Sigma(\Gamma_n)$, we have
  \begin{equation*}
    \dima(\mu,\pi(\iii))=\dima(\nu,\iii).
  \end{equation*}
\end{lemma}
\begin{proof}
Let $\iii=\iii_1\kkk\iii_2\kkk\cdots\in\Sigma(\Gamma_n)$, where $\iii_j\in\Sigma_{n-|k|}$ for all $j\in\N$.  It suffices to verify the following: There is a constant $C>0$, such that
\begin{equation*}
  p_{\iii|_m}\leq \mu(B(\pi(\iii),r))\leq Cp_{\iii|_m},
\end{equation*}
for all $r>0$, where $m\in\N$ is the unique integer satisfying $r_{\iii|_m}< r \leq r_{\iii|_{m-1}}$. Note that the lower bound is trivial, since $\varphi_{\iii|_m}(X)\subset B(\pi(\iii),r)$, so it remains to prove the upper bound. Let $k$ denote the unique integer satisfying
\begin{equation*}
  r_{\iii_1}r_{\iii_2}\ldots r_{\iii_k}r_{\kkk}^k< r \leq r_{\iii_1}r_{\iii_2}\ldots r_{\iii_{k-1}}r_{\kkk}^{k-1},
\end{equation*}
and note that $\iii|_{m}=\iii_1\kkk\iii_2\kkk\ldots \iii_{k}\kkk\lll$, for some $\lll\in \Sigma_*$ satisfying $|\lll|\leq n$. Applying Lemma \ref{lemma:osc-ball-measure}, gives
\begin{equation*}
  \mu(B(\pi(\iii),r))\leq p_{\iii_1}\cdots p_{\iii_{k-1}}p_\kkk^{k-2}=\frac{1}{p_{\kkk}^2p_{\iii_k}p_{\lll}}p_{\iii_1}\cdots p_{\iii_{k}}p_\kkk^{k}p_{\lll}\leq Cp_{\iii|_m},
\end{equation*}
where $C=p_{\min}^{-2n-|k|}$.
\end{proof}

We conclude by combining the results of this section to prove Theorems \ref{thm:main-doubling} and \ref{thm:main-spectrum} starting with the latter.
\begin{proof}[Proof of Theorem \ref{thm:main-spectrum}]
  Recall that 
  \[D_{\alpha}^\A\subset U_{\alpha}^{\A}\subset U_{\alpha}^{\loc}\] 
  so $\dimh D_{\alpha}^\A\le\dimh U_{\alpha}^\A\le \overline{f}(\alpha)$, by using the fact that Theorem \ref{thm:classical-formalism} holds for self-similar measures under the OSC. For $\alpha\in(\alpha_{\min},\alpha_{\max})$ and $\varepsilon>0$ let $\Omega_\varepsilon(\alpha)$ be as in \eqref{eq:Omega-construction} and recall that $\Omega_\varepsilon(\alpha)\subset \Sigma(\Gamma_n)$. By Lemma \ref{lemma:subset-non-doubling}, we have $\pi(\Omega_\varepsilon(\alpha))\subset D_{ \alpha}^\A$ and the OSC implies that $\dimh\pi(\Omega_\varepsilon(\alpha))=\dimh \Omega_\varepsilon(\alpha)$. Letting $\varepsilon\downarrow 0$ and applying Lemma \ref{lemma:Omega-properties}, we arrive at $\dimh D_\alpha^\A\geq \barf(\alpha)$. 
  
  If $\alpha=\alpha_{\min}$, it suffices to show by Lemma \ref{lemma:alpha_min-lower-bound}, that $\pi(\Sigma(\Lambda_{\min}))\subset D_\alpha^\A(\mu)$. For this, let $\iii\in \Sigma(\Lambda_{\min})$, and note that by Theorem \ref{thm:doubling-iff-leq-alpha_max}, it is enough to show that $\dima(\mu,\pi(\iii))\leq \alpha_{\min}$. Let $0<r<R$ and note that for any $\jjj\in \Sigma(x,R)$, we have
  \begin{equation*}
    p_{\jjj}=\prod_{j=1}^{|\jjj|}r_j^{\frac{\log p_j}{\log r_j}}\leq r_\jjj^{\alpha_{\min}}\leq R^{\alpha_{\min}}.
  \end{equation*}
  Moreover, by choosing $k\in\N$, such that $r_{\iii|_{k}}< r\leq r_{\iii|_{k-1}}$, we see that
  \begin{equation*}
    \mu(B(x,r))\geq p_{\iii|_{k}}=\prod_{i=1}^{k}r_i^{\frac{\log p_i}{\log r_i}}=r_{\iii|_k}^{\alpha_{\min}} \geq r_{\min} r^{\alpha_{\min}}.
  \end{equation*}
  Therefore,
  \begin{equation*}
    \frac{\mu(B(x,R))}{\mu(B(x,r))}\leq \frac{\sum_{\jjj\in\Sigma(x,R)}p_\jjj}{p_{\iii|_k}}\leq \frac{M}{r_{\min}}\left(\frac{R}{r}\right)^{\alpha_{\min}},
  \end{equation*}
  implying $\dima(\mu,\pi(\iii))\leq \alpha_{\min}$.

  The second claim follows easily from Proposition \ref{prop:exact-dim-non-doubling}, since by letting $\nu$ be the canonical self-similar measure satisfying $\dimh \nu=\dimh X$, we have that $\nu(D_\infty^\A)=1$, and thus $\dimh D_\infty^\A\geq \dimh\nu=\dimh X$. Finally, it was shown in Theorem \ref{thm:doubling-iff-leq-alpha_max}, that $D_\alpha^\A=\emptyset$ for $\alpha\not\in[\alpha_{\min},\alpha_{\max}]\cup\{\dima\mu\}$.
\end{proof}
\begin{proof}[Proof of Theorem \ref{thm:main-doubling}]
  Claim \eqref{thm:main-1} follows from Theorem \ref{thm:main-spectrum}, since $D_{\alpha(0)}^\A(\mu)\subset D(\mu)$, and since $\dimh D_{\alpha(0)}^\A(\mu)=f(0)=\dimh X$. Claim \eqref{thm:main-2} is Corollary \ref{cor:main-2}.
\end{proof}

\begin{remark}\label{rem:alpha_max}
    Excluding $\alpha_{\max}$ from Theorem \ref{thm:main-spectrum} is necessary: In Example \ref{ex:Ber} we saw that for the $(p,1-p)$-Bernoulli measure on $[0,1]$, $\dima(\mu,0)=\alpha_{\max}=\frac{\log p}{-\log 2}$. However, it is not too difficult to see that $0$ is the only element in $D_{\alpha_{\max}}^\A$. Indeed, if $0<x=\pi(\iii)\le 1$ is a dyadic rational, then $\dima(\mu,x)=\alpha_{\min}$. On the other hand, if $x$ is not a dyadic rational and the number of consecutive zeros in $\iii$ is unbounded, $\mu$ is not doubling at $x$ and thus $\dima(\mu,x)=\infty\neq\alpha_{\max}$. Finally, if the number of consecutive zero's in $\iii$ is bounded by $N\in\N$, the computations in Example \ref{ex:Ber} may be refined by replacing all $p^{-m}$ terms by $C p^{-m(1-1/N)}(1-p)^{-m/N}$ for a suitable constant $C$. This yields an upper bound $\dima(\mu,x)\le(1-\tfrac1N)\alpha_{\max}+\tfrac1N\alpha_{\min}<\alpha_{\max}$. There are also examples at the other extreme where $\dimh D_{\alpha_{\max}}^\A=\dimh X$. For example take a self-similar set, where $p_{\kkk}=p_{\min}^{|\kkk|}$. This ensures that $\alpha_{\max}(\Gamma_n)=\alpha_{\max}$ for all $n$ in Lemma \ref{lemma:exists-point-with-dim} which allows us to extend Lemma \ref{lemma:Omega-properties} for $\alpha=\alpha_{\max}$. The details are left to the interested reader.
\end{remark}

\begin{remark}
  The multifractal spectrum for the classical pointwise dimension has been studied in a variety of more general settings such as for self-similar sets under less restrictive separation conditions \cite{LauNgai1999, FengLau2009, Feng2012, BarralFeng2021}. For self-affine sets the situation is more complicated, but some results are known \cite{JordanRams2011,Fraser2021}. In \cite{Anttila2022} an exact dimensionality property for the pointwise Assouad dimension of certain self-affine carpets was established. It would be of great interest to study the pointwise doubling property, and more generally, the multifractal spectrum of the pointwise Assouad dimension in these contexts.
\end{remark}
 
\section*{Acknowledgements}
We are grateful to Thomas Jordan and Alex Rutar for valuable conversations on the topics of this paper. We also thank Alex Rutar for introducing us to the methods in \cite{Rutar2023a}.

\bibliography{bibliography}
\bibliographystyle{abbrv}

\end{document}